\documentclass[10pt,reqno]{amsart}

\usepackage[utf8]{inputenc}
\usepackage{amssymb, amsmath, amsfonts, amsthm, mathrsfs, tikz-cd, mathtools}
\usepackage{graphicx}
\usepackage[url=false,style=alphabetic]{biblatex}

\newtheorem{thm}{Theorem}
\newtheorem{prop}[thm]{Proposition}
\newtheorem{lem}[thm]{Lemma}
\newtheorem{cor}[thm]{Corollary}

\theoremstyle{definition}
\newtheorem{definition}[thm]{Definition}
\newtheorem{example}[thm]{Example}

\theoremstyle{remark}
\newtheorem{remark}[thm]{Remark}

\numberwithin{equation}{section}

\bibliography{references}

\graphicspath{{./images}}

\DeclareMathOperator{\GenCon}{GenCon}

\newcommand{\Z}{\mathbb{Z}}
\newcommand{\R}{\mathbb{R}}

\begin{document}

\begin{abstract}
    We study the relationship between a notion of medium-scale Ricci curvature for finitely generated groups and that of hyperbolicity in the sense of Gromov. We give an example of a generating set that gives zero curvature with positive density for the free group of rank 2. We prove that, by making the radius used in computing the curvature sufficiently large, we can always have negative curvature outside of a ball in non-elementary hyperbolic groups. On the other hand, we give an example of a group which has negative curvature for all non-identity points but is not hyperbolic.
\end{abstract}

\title{Medium-Scale Ricci Curvature for Hyperbolic Groups}
\author{Andrew Keisling}
\date{January 2021}

\maketitle
\tableofcontents

\section*{Acknowledgments}
The author would like to thank Thang Nguyen for providing helpful mentorship and feedback throughout this project. They would also like to thank the University of Michigan Math REU for providing the working conditions that made this project possible. The author was partially supported by NSF grant DMS-2003712.

\section{Introduction}

In Riemannian geometry, Ricci curvature measures the difference of the average distance between points on infinitesimal spheres and the distance between their centers. In \cite{Oll09}, Ollivier generalized this intuition to more general metric spaces equipped with Markov chains with his \emph{transportation curvature}. Here we focus on the \emph{conjugation curvature} for finitely generated groups introduced by Bar-Natan, Duchin, and Kropholler in \cite{BDK20}. They called this curvature \emph{medium-scale} because, rather than being based on infinitesimal or global properties, it depends on some fixed comparison radius.

In this paper, following the notation introduced by Bar-Natan et al., we use $\kappa_k^\mathcal{S}(g)$ to denote the $k$-spherical conjugation curvature at the group element $g$. As a special case, we write $\kappa(g)$ to denote the $1$-spherical conjugation curvature. It follows from the definition that this type of curvature depends on the choice of generating set of the group. 

The first question that we seek to answer is whether the free group of rank 2 must have negative curvature for all but finitely many points with respect to any generating set. Such a result would answer a question asked by Nguyen and Wang in \cite[Question 5.3]{NW20} and show whether torsion-free, non-elementary hyperbolic groups can have infinitely many points with non-negative curvature. We show that they indeed can by proving the following theorem:

\begin{thm} \label{intro zero curvature dense}
    There exists a generating set of the free group $F_2 = \langle a,b \rangle$ such that $F_2$ has zero curvature with positive density. $F_2$ also has negative curvature with positive density with respect to this generating set.
\end{thm}
By positive density, we mean that there exist $\epsilon>0$, $N\geq0$ such that the fraction of points satisfying the desired property in the ball of radius $n$ is bounded below by $\epsilon$ for all $n\geq N$.

% It is natural to ask whether there also exists a generating set of $F_2$ giving positive curvature for infinitely many points. We do not have the answer to this question, but we do show that we can have infinitely many points with positive curvature in virtually free groups.

% \begin{thm} \label{intro positive virtual}
%     There exists a generating set of the virtually free group $\Z_2*\Z_3 = \langle a,b|a^2=b^3=1\rangle$ and a sequence $(g_n)_{n=1}^\infty$ that tends to $\infty$ in $\Z_2*\Z_3$ such that $\kappa(g_n) > 0$ for all $n$.
% \end{thm}

In addition to this result for $F_2$, we also explore the relationship between conjugation curvature and non-elementary hyperbolic groups in general. By \cite[Theorem 20]{BDK20}, there exist non-elementary hyperbolic groups with sequences of elements $(g_n)_{n=1}^\infty$ such that $\kappa(g_n)>0$ for all $n$. We prove the following theorem to show that, by using a sufficiently large comparison radius $k$ to compute the curvature, we can have at most finitely many points with non-negative curvature in any non-elementary hyperbolic group:

\begin{thm} \label{intro negative curvature for hyperbolic}
    Let $G$ be a non-elementary hyperbolic groups with some fixed generating set $S$. Then there exists $K \geq 1$ such that, for every $k\geq K$ and $g \in G$ outside the ball of radius $4k$, we have $\kappa_k^{\mathcal{S}}(g)<0$.
\end{thm}

For this result for hyperbolic groups to be meaningful, it is significant that there exist infinite groups that have positive $k$-conjugation curvature for infinitely many points for arbitrarily large $k$. Although the existence of such groups is non-obvious, Kropholler and Mallery proved in \cite[Theorem 4.1]{KM20} that the discrete Heisenberg group $H(\Z) = \langle a,b|[a,b]~\text{is central}\rangle$ satisfies this property.

On the other hand, we show that negative curvature for all non-identity points does not imply that a group is hyperbolic by proving the following counterexample:

\begin{thm} \label{intro negative non-hyperbolic}
    The group $\Z*\Z^2 = \langle t,a,b|ab=ba\rangle$ is not hyperbolic, but there exists a generating set such that, for all $g\in\Z*\Z^2-\{1\}$, we have $\kappa(g)<0$.
\end{thm}

Due to the result of Theorem \ref{intro zero curvature dense}, it is natural to ask whether every non-elementary hyperbolic group can have non-negative curvature for a positive density of points with respect to some generating set. A similar interesting question for future research is whether $F_2$ can have strictly positive curvature for infinitely many elements with respect to some generating set. Even if this is not the case, it seems possible that a similar proof to that of Theorem \ref{intro zero curvature dense} can be applied to the virtually free group $\Z_2*\Z_3 = \langle a,b|a^2=b^3=1\rangle$ to give positive curvature with positive density. If this turns out to be true, then another interesting question to study is whether the existence of order 2 elements gives a necessary and sufficient condition for strictly positive curvature for infinitely elements in non-elementary hyperbolic groups. One reason that this might be the case is that a generating set of odd size must contain an element of order 2 (in the definition of conjugation curvature, we use generating sets that are closed under inversion and never contain the group identity). Note that the generating set used to give positive curvature in a hyperbolic group in \cite[Theorem 20]{BDK20} has odd size, while $F_2$ has no elements of order 2 because it is torsion-free.

This paper is organized as follows: we begin by reviewing the definitions of conjugation curvature and hyperbolicity in Section \ref{background}, going over some simple  examples and proving lemmas that will help with the calculation of curvature. In Section \ref{free and virtually free}, we give examples of generating sets giving negative curvature for $F_2$ before proving Theorem \ref{intro zero curvature dense}. We conclude in Section \ref{relationship with hyperbolicity} by proving Theorems \ref{intro negative curvature for hyperbolic} and \ref{intro negative non-hyperbolic}.

\section{Background} \label{background}

In this section we review the definitions of conjugation curvature and Gromov's $\delta$-hyperbolicity, looking at some specfic examples. We then prove some preliminary results that will help us compute the curvature.

\subsection{Conjugation Curvature}

Let $G$  be a group that is generated by some finite set $S$, which we assume to be closed under inversion and not contain the group identity. Define the \emph{word length} $|g|_S$ of an element $g \in G$ as the minimum number of generators $g_1,\ldots ,g_n \in S$ such that $g=g_1\ldots g_n$. We call $g_1\ldots g_n$ a \emph{word} of generators and a \emph{spelling} of $g$. When $n=|g|_S$, we further call $g_1\ldots g_n$ a \emph{geodesic spelling} of $g$ (such spellings need not be unique). In general, a \emph{geodesic word} of generators is a word of generators such that no other word spelling the same group element has less generators.

The norm $|\cdot|_S:G\rightarrow\Z_{\geq 0}$ induces a left-invariant metric $d_S:G\times G\rightarrow \Z_{\geq 0}$ defined by $d_S(g,h) = |h^{-1}g|_S$, so that $d_S(g,1) = |g|_S$ where $1$ denotes the group identity. Although the metric space $(G,d_S)$ depends on the choice of generating set $S$, we will henceforth not include the subscript $S$ because the generating set will be clear from context.

Let $(X,d)$ be a metric space. Then we call $(X,d)$ \emph{geodesic} if any two points in $X$ can be connected by a continuous curve $c:[a,b]\subset\R\rightarrow X$ such that $d(c(t_1),c(t_2)) = |t_1-t_2|$ for all $t_1,t_2\in[a,b]$. 

Given a finitely generated group $G$ and finite generating set $S$, define the \emph{Cayley graph} of $(G,S)$ as the graph whose vertices are the elements of $G$ and whose edges connect the points $g,h\in G$ if $d(g,h) = 1$. This graph can be equipped with the natural graph metric such that the length of each edge is 1, and the distance between two points is the length of the shortest path connecting them. With this metric, the Cayley graph is a geodesic metric space such that the distance between two vertices is exactly the word distance between the corresponding group elements. An example of a Cayley graph is shown in Figure \ref{Z^2Cayley} on page \pageref{Z^2Cayley}.

\begin{figure}
    \centering
    \includegraphics[scale = 0.25]{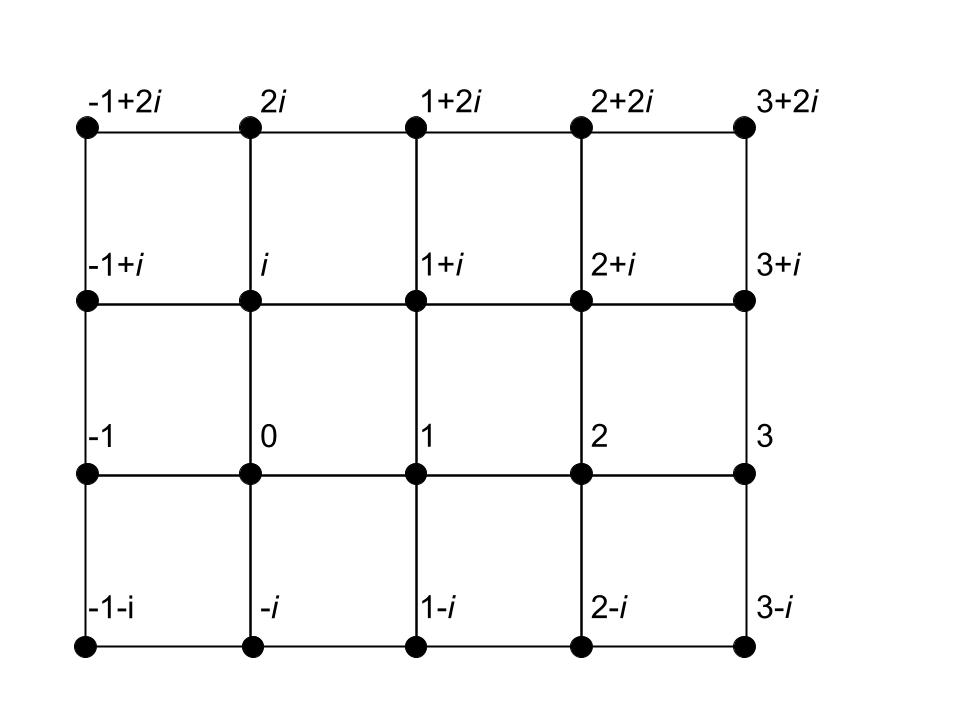}
    \caption{A portion of the Cayley graph of the additive group $\Z^2$ with respect to the generating set $S\coloneqq \{1,-1,i,-i\}$.}
    \label{Z^2Cayley}
\end{figure}

By combining this geometric structure for a finitely generated group $G$ with the algebraic structure of $G$, we can define a notion of curvature which shares many properties with Ricci curvature for Riemannian manifolds. For $k \geq 0$ and $g \in G$, define the ball and sphere of radius $k$ centered at $g$ to respectively be
\begin{gather*}
    B_k(g) \coloneqq \{h \in G|d(g,h)\leq k\}, \\
    S_k(g) \coloneqq \{h \in G|d(g,h) = k\}.
\end{gather*}
Usually we care about balls and spheres centered at the group identity $1$, which we write as $B_k$ and $S_k$ respectively. In particular, note that we have $S_1 = S$. Given any distinct $g,h \in G$ and $k\geq 1$, define the \emph{$k$-spherical comparison distance} between $g$ and $h$ to be
\begin{equation*}
    \mathcal{S}_k(g,h) \coloneqq \frac{1}{|S_k|}\sum_{s\in S_k}d(gs,hs).
\end{equation*}
The quantity $\mathcal{S}_k(g,h)$ measures the average distance between corresponding points on the spheres of radius $k$ about $g$ and $h$, where the correspondence is by left-translation. Comparing this quantity to the distance between $g$ and $h$, we can write the definition of the type of curvature that we are concerned with:
\begin{definition}\cite{BDK20}
    The \emph{$k$-spherical conjugation curvature} between $g,h\in G$ is
    \begin{equation*}
        \kappa_k^{\mathcal{S}}(g,h) \coloneqq \frac{d(g,h)-\mathcal{S}_k(g,h)}{d(g,h)} = \frac{d(g,h)-\frac{1}{|S_k|}\sum_{s\in S_k}d(gs,hs)}{d(g,h)}.
    \end{equation*}
\end{definition}
Notice that we get positive curvature when corresponding points on $S_k(g)$ and $S_k(h)$ are closer, on average, than $g$ and $h$, a property that agrees with classical Ricci curvature. By left-invariance of our metric, it suffices to consider the conjugation curvature between a point $g \in G-\{1\}$ and the group identity $1$, which we write as $\kappa_k^{\mathcal{S}}(g) \coloneqq \kappa_k^{\mathcal{S}}(g,1)$. For the case of $k=1$, we use the notation $\kappa(g) \coloneqq \kappa_1^{\mathcal{S}}(g)$. Note that
\begin{equation*}
    \mathcal{S}_1(g,1) = \frac{1}{|S_1|}\sum_{s\in S_1}d(gs,1s) = \frac{1}{|S|}\sum_{s\in S}|s^{-1}gs| \eqqcolon \GenCon(g),
\end{equation*}
where we use the notation $\GenCon(g)$ to emphasize the fact that this quantity is the average word length of the conjugation of $g$ by all the generators in $S$ (hence the name conjugation curvature). Therefore,
\begin{equation*}
    \kappa(g) = \frac{|g|-\GenCon(g)}{|g|},
\end{equation*}
and in order to find the sign of the curvature it suffices to evaluate whether conjugation by generators on average increases, decreases, or preserves the word length of $g$.

\begin{example}
    Suppose that $G$ is an abelian group. Then for all $g\in G-\{1\}$ we have
    \begin{align*}
        \kappa(g) &= \frac{|g|-\GenCon(g)}{|g|} = \frac{|g|-\frac{1}{|S|}\sum_{s\in S}|s^{-1}gs|}{|g|} \\&= \frac{|g|-\frac{1}{|S|}\sum_{s\in S}|gs^{-1}s|}{|g|} = \frac{|g|-\frac{1}{|S|}\sum_{s\in S}|g|}{|g|} \\&= \frac{|g|- |g|}{|g|} = 0.
    \end{align*}
    Therefore, abelian groups are everywhere ``flat'' with respect to the conjugation curvature.
\end{example}

\subsection{Hyperbolic Groups}

The notion of a \emph{hyperbolic group} was introduced by Gromov in \cite{Gro87}. This is a group that, when equipped with a word metric, satisfies the following:
\begin{definition}\cite{Gro87}
    A metric space $(X,d)$ is called \emph{$\delta$-hyperbolic} for $\delta\geq 0$ if, for all points $x,y,z,w \in X$, the \emph{four-point condition}
    \begin{equation*}
        (x,y)_w \geq \min\{(x,z)_w,(y,z)_w\}-\delta
    \end{equation*}
    is satisfied. Here the quantity $(p,q)_r$ denotes the \emph{Gromov product} of $p,q,r\in X$ defined by
    \begin{equation*}
        (p,q)_r \coloneqq \frac{1}{2}(d(p,r)+d(r,q)-d(p,q)).
    \end{equation*}
\end{definition}
Gromov proved in \cite{Gro87} that, if a group is $\delta$-hyperbolic with respect to some finite generating set $S$, then with respect to another finite generating set $S'$ it is $\delta'$-hyperbolic for some $\delta'\geq 0$. Therefore, we can speak of groups being hyperbolic, but in order to speak of a group being $\delta$-hyperbolic we must first fix a generating set.

Another useful formulation of hyperbolicity for geodesic metric spaces (such as a Cayley graph with a graph metric) is based on \emph{$\delta$-thin triangles}, which are defined as follows: Given a geodesic triangle $\Delta$ consisting of the vertices $x$, $y$, and $z$ connected by geodesic segments, there exists a map $\tau$ from $\Delta$ to a tripod whose endpoints are $\tau(x)$, $\tau(y)$, and $\tau(z)$ such that $d(p_1,p_2) = d(\tau(p_1),\tau(p_2))$ whenever $p_1$ and $p_2$ lie on the same edge of $\Delta$. The length of the branch of this tripod that ends in $\tau(z)$ is exactly the product $(x,y)_z$, and an analogous result holds for the other two branches. We call $\tau$ a \emph{triangle-tripod transformation}, and its image is unique up to isometry. We call $\Delta$ a $\delta$-thin triangle if, whenever $\tau(p) = \tau(q)$ for two points $p,q\in\Delta$, we have $d(p,q)\leq\delta$ (see Figure \ref{thin-triangle} on page \pageref{thin-triangle}).

\begin{figure}
    \centering
    \includegraphics[scale = 0.27]{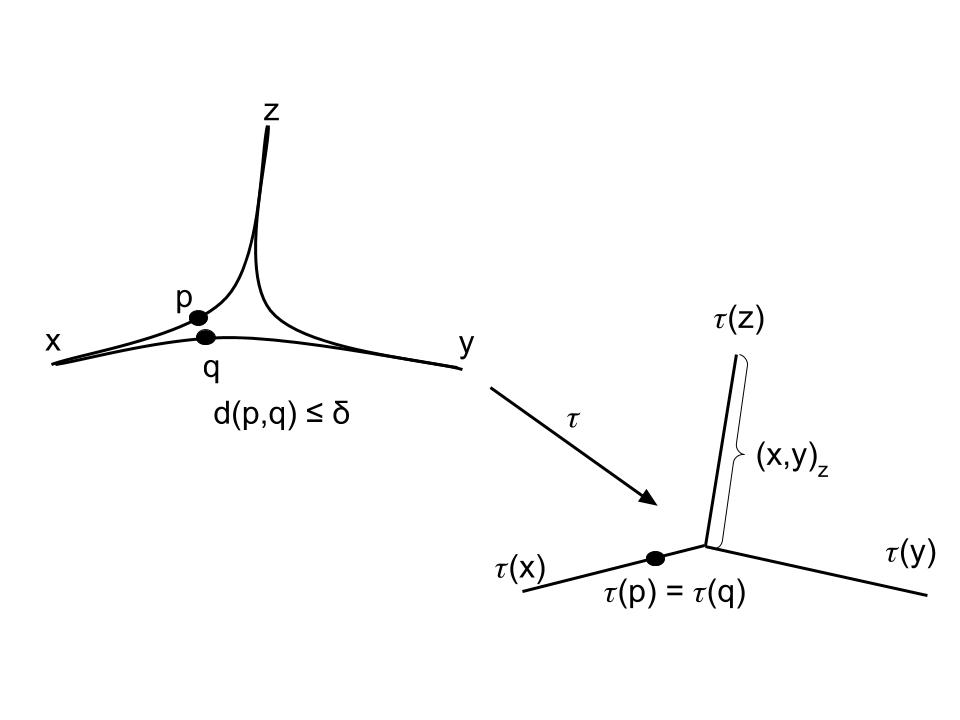}
    \caption{A $\delta$-thin triangle $\Delta$ in a geodesic metric space, along with its image under a triangle-tripod transformation $\tau$. Any points that map to the same point under $\tau$ must lie within a distance of $\delta$ in $\Delta$.}
    \label{thin-triangle}
\end{figure}

According to \cite[Proposition 6.3.C]{Gro87}, all geodesic triangles in a $\delta$-hyperbolic space are $2\delta$-thin, and if all geodesic triangles in a geodesic metric space are $\delta$-thin then the space is $2\delta$-hyperbolic. In this paper, when we write that a metric space is $\delta$-hyperbolic, the constant $\delta$ is determined by the four-point condition.

\begin{example}
    Consider a metric space $(X,d)$ which is a tree. Any geodesic triangle in $X$ will either be a tripod or a  single geodesic segment containing all three of its vertices. Therefore, all geodesic triangles in $X$ are $0$-thin and $X$ is a $0$-hyperbolic metric space. 
\end{example}

\subsection{Computing GenCon}

The following basic results will be used throughout this paper to transform the problem of computing the quantity $\GenCon(g)$ that appears in the formula for $\kappa(g)$ to that of finding spellings of $g$ that begin and end with certain generators:

\begin{lem}  \label{shortening lemma}
    Let $S$ be a generating set for the group $G$. Let $g \in G$ and $s \in S$. Then $|sg| = |g|-1$ if and only if there exists a geodesic spelling $g_1\ldots g_n$ of $g$ with respect to $S$ such that $g_1 = s^{-1}$. Similarly, $|gs| = |g|-1$ if and only if there exists a geodesic spelling $g_1\ldots g_n$ of $g$ with respect to $S$ such that $g_n = s^{-1}$.
\end{lem}
\begin{proof}
    Assume that a geodesic spelling $s^{-1}g_2\ldots g_n=g$ exists. Then $|sg| = |ss^{-1}g_2\ldots g_n| = |g_2\ldots g_n| \leq |g|-1$, and from the triangle inequality we have $|sg| \geq |g|-1$, so $|sg| = |g|-1$.
    
    Conversely, assume that $|sg| = |g|-1$. Then there exists a geodesic spelling $h_1\ldots h_{n-1}$ of $sg$, where $|g| = n$. Consider the word $s^{-1}h_1\ldots h_{n-1}$, which is equal to $s^{-1}sg = g$. Because this word consists of $n$ generators, it is a geodesic spelling of $g$, and it begins with $s^{-1}$ as desired. 
    
    Note that we have $|g|$ = $|g^{-1}|$ for all $g\in G$ because any geodesic spelling $g_1\ldots g_n$ of $g$ induces a spelling $g_n^{-1}\ldots g_1^{-1}$ of $g^{-1}$ with the same number of generators. The proof for $|gs|=|g|-1$ therefore follows from noting that $|gs| = |(gs)^{-1}| = |s^{-1}g^{-1}|$ and applying the result for when $|s^{-1}g^{-1}| = |g^{-1}|-1 = |g|-1$.
\end{proof}

\begin{lem}  \label{preserving lemma}
    Let $S$, $G$, $s$, and $g$ be as in the statement of Lemma \ref{shortening lemma} with $|g| = n$. Then $|sg| = |g|$ if and only if there exists a spelling $g_1\ldots g_{n+1}$ of $g$ with respect to $S$ such that $g_1 = s^{-1}$ and $g_2\ldots g_{n+1}$ is a geodesic word. Similarly, $|gs|=|g|$ if and only if there exists a spelling $g_1\ldots g_{n+1}$ of $g$ with respect to $S$ such that $g_{n+1} = s^{-1}$ and $g_1\ldots g_n$ is a geodesic word.
\end{lem}
\begin{proof}
    Assume that a spelling $s^{-1}g_2\ldots g_{n+1}=g$ exists with $g_2\ldots g_{n+1}$ a geodesic word. Then $|sg| = |ss^{-1}g_2\ldots g_{n+1}| = |g_2\ldots g_{n+1}| = n = |g|$.
    
    Conversely, assume that $|sg| = |g|$. Then there exists a geodesic spelling $h_1\ldots h_n$ of $sg$, so $g = s^{-1}sg = s^{-1}h_1\ldots h_n$. This is indeed a spelling of $g$ that is one letter longer than a geodesic spelling and that begins with $s^{-1}$ as desired. 
    
    As in the proof of Lemma \ref{shortening lemma}, the proof for $|gs|=|g|$ follows from applying the result for $|sg|=|g|$ to the inverse of $gs$.
\end{proof}

\begin{remark}  \label{extending remark}
    By the triangle inequality for the word metric, $|g|-|s| \leq |sg| \leq |g|+|s|$, so $|sg| = |g|-1$, $|g|$, or $|g|+1$ because $|s| = 1$. Lemmas \ref{shortening lemma} and \ref{preserving lemma} give necessary and sufficient conditions for the first two of these three possibilities, so if the hypotheses of neither of these lemmas hold then $|sg| = |g|+1$. Similarly, if the necessary and sufficient conditions for $|gs|=|g|-1$ and $|gs|=|g|$ both fail to hold, then $|gs|=|g|+1$.
\end{remark}

\section{Curvature of Free Group} \label{free and virtually free}

Throughout this section, let $F_2$ denote the free group of rank 2 with the presentation $\langle a,b \rangle$. We begin by giving examples of generating sets of $F_2$ for which all non-identity points have negative curvature. We then prove Theorem \ref{intro zero curvature dense}.

\subsection{Negative Curvature Examples in Free Group} \label{negative curvature in free}

\begin{example}
    Consider the generating set $S_\alpha\coloneqq\{a,a^{-1},b,b^{-1}\}$ of $F_2$. Because the only way to shorten a word in these generators is by applying free reductions, any two words can only represent the same group element $g\in F_2-\{1\}$ if they differ in length by an even number of generators. Therefore, by Lemma \ref{preserving lemma}, there exist no generators $s\in S_\alpha$ or group elements $g\in F_2-\{1\}$ such that $|sg|=|g|$ or $|gs|=|g|$. Moreover, because this group is free, every group element has only one geodesic spelling with respect to $S_\alpha$. For each $g\in F_2-\{1\}$, we therefore have by Lemma \ref{shortening lemma} that there exists only one $s\in S_\alpha$ such that $|sg| = |g|-1$ and only one $s\in S_\alpha$ such that $|gs| = |g|-1$. For all other $s\in S_\alpha$ we have $|sg|=|g|+1$ and $|gs|=|g|+1$, so we can conclude that 
    \begin{equation*}
        \kappa(g) = \frac{|g|-\GenCon(g)}{|g|} = \frac{|g|-\frac{1}{4}\sum_{s\in S_\alpha} |s^{-1}gs|}{|g|} = \frac{|g|-\frac{1}{4}(4|g|+4)}{|g|} = -\frac{1}{|g|}.
    \end{equation*}
    The Cayley graph of $(F_2,S_\alpha)$ is shown in Figure \ref{F2Cayley} on page \pageref{F2Cayley}. Note that this graph is a tree, proving that $F_2$ is a hyperbolic group.
\end{example}

\begin{figure}
    \centering
    \includegraphics[scale = 0.25]{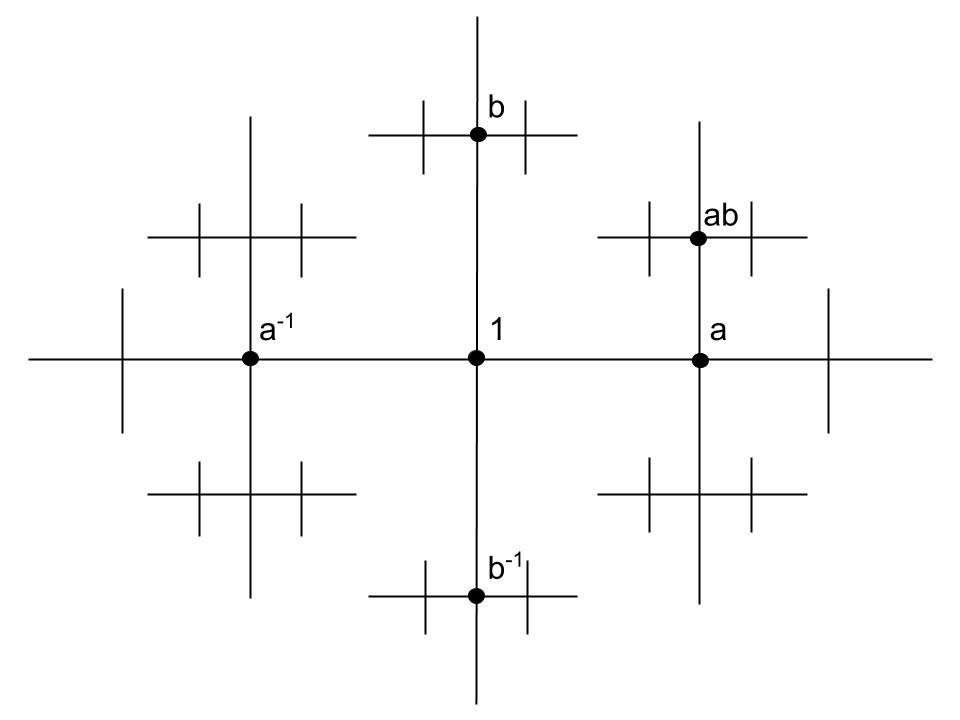}
    \caption{The Cayley graph of the group $F_2 = \langle a,b \rangle$ with respect to the generating set $S_\alpha \coloneqq \{a,a^{-1},b,b^{-1}\}$.}
    \label{F2Cayley}
\end{figure}

Now consider the generating set $S_\beta\coloneqq \{a,a^{-1},b,b^{-1},aba,a^{-1}b^{-1}a^{-1}\}$. We have the following result:

\begin{prop} \label{negative curvature example}
     With respect to the generating set $S_\beta$, we have $\kappa(g) < 0$ for all $g\in F_2-\{1\}$.
\end{prop}
\begin{proof}
    Let $w = aba$ so that $w^{-1} = a^{-1}b^{-1}a^{-1}$. With respect to the generators in $S_\beta$, $F_2$ has the presentation $\langle a,b,w|abaw^{-1}=1\rangle$. Suppose that we have $g_1\ldots g_n=h_1\ldots h_{n+1}$ for some generators $g_1,\ldots,g_n,h_1,\ldots,h_{n+1}\in S_\beta$. Then $h_{n+1}^{-1}\ldots h_1^{-1}g_1\ldots g_n=1$, so this word with $2n+1$ letters can be reduced to the identity by applying free reductions and removing all occurrences of any inverse or cyclic permutation of $abaw^{-1}$. Because all of these operations change the number of letters by an even number and $2n+1$ is odd, we have a contradiction. Therefore, given any word of generators in $S_\beta$, there does not exist a word with exactly one more generator that represents the same group element in $F_2$. By Lemma \ref{preserving lemma}, this implies that there do not exist generators $s\in S_\beta$ or group elements $g\in F_2-\{1\}$ such that $|sg|=|g|$ or $|gs|=|g|$.

    For $s\in S_\alpha = \{a,a^{-1},b,b^{-1}\}$, let $A_s\subset F_2$ be the set of all $g\in F_2$ such that $g$ begins with $s$ in its geodesic spelling with respect to $S_\alpha$. Every non-identity element of $F_2$ belongs to exactly one of the four sets $A_a$, $A_{a^{-1}}$, $A_b$, or $A_{b^{-1}}$, which each correspond with one of the four primary branches stemming from the point $1$ in Figure \ref{F2Cayley}. We claim that, for all $g\in A_a$, every geodesic spelling of $g$ with respect to $S_\beta$ must begin with either $a$ or $w$. Note that $a$ and $w$ are the only two generators in $S_\beta$ that begin with $a$ when spelled with respect to $S_\alpha$ and thus lie in $A_a$. Suppose that there exists $g\in A_a$ with a geodesic spelling $g= g_1\ldots g_n$ such that $g_1\neq a$ and $g_1\neq w$. Then for some $1\leq k<n$ we have $g_1\ldots g_k\notin A_a\cup \{1\}$ and $(g_1\ldots g_k)g_{k+1}\in A_a$. The only non-identity point that is not in $A_a$ for which right multiplication by a generator gives a point in $A_a$ is $b^{-1}a^{-1}$, for which we have $(b^{-1}a^{-1})w=a$. Therefore, we must have $(g_1\ldots g_k)g_{k+1}=a$, contradicting our claim that the spelling $g_1\ldots g_n$ was geodesic. Hence, the first letter of a geodesic spelling of any element of $A_a$ must be either $a$ or $w$ as desired. An identical argument shows that a geodesic spelling of an element of $A_{a^{-1}}$ must begin with either $a^{-1}$ or $w^{-1}$.
    
    We now claim that every geodesic spelling of an element of $A_b$ must begin with either $b$ or $a^{-1}$. Suppose that $g\in A_b$ has a geodesic spelling $g= g_1\ldots g_n$ such that $g_1\neq b$. Then again there exists $1\leq k<n$ such that $g_1\ldots g_k\notin A_b\cup\{1\}$ and $(g_1\ldots g_k)g_{k+1}\in A_b$. The only non-identity point that is not in $A_b$ for which right multiplication by a generator gives a point in $A_b$ is $a^{-1}$, for which we have $a^{-1}w = ba$. Therefore, $g_1\ldots g_k=a^{-1}$. Because the spelling $g_1\ldots g_n$ is geodesic, we must have $k=1$ and $g_1=a^{-1}$. Hence, the first letter of a geodesic spelling of an element of $A_b$ must be either $b$ or $a^{-1}$. An identical argument shows that a geodesic spelling of an element of $A_{b^{-1}}$ must begin with either $b^{-1}$ or $a$.
    
    We thus conclude by Lemma \ref{shortening lemma} that, for all $g\in F_2-\{1\}$, there exist at most two distinct $s\in S_\beta$ such that $|sg|=|g|-1$. By Remark \ref{extending remark} we must have $|sg|=|g|+1$ for all other $s$. By applying the results of the previous paragraphs to the inverses of group elements, we also get that there exist at most two distinct $s\in S_\beta$ such that $|gs|=|g|-1$.
    
    Direct computation shows that $\kappa(g) < 0$ for $g=ba$ and $g=b^{-1}a^{-1}$, and whenever $g$ is a generator in $S_\beta$. For all other $g\in F_2-\{1\}$, note that $g$ is in $A_{s_0}$ for some $s_0\in S_\alpha=\{a,a^{-1},b,b^{-1}\}$ if and only if $gs'$ is also in $A_{s_0}$ for all generators $s'\in S_\beta$. Hence, the two distinct $s\in S_\beta$ for which it is possible to have $|sg|=|g|-1$ are the same as the two distinct $s\in S_\beta$ for which it is possible to have $|sgs'|=|gs'|-1$. Therefore,
    \begin{align*}
        \kappa(g) &= \frac{|g|-\GenCon(g)}{|g|} = \frac{|g|-\frac{1}{6}\sum_{s\in S_\beta}|s^{-1}gs|}{|g|} \\
        &= \frac{\frac{1}{6}\sum_{s\in S_\beta}(|g|-|gs|) + \frac{1}{6}\sum_{s\in S_\beta}(|gs|-|s^{-1}gs|)}{|g|}\\ &\leq \frac{\frac{1}{6}(-2)+\frac{1}{6}(-2)}{|g|} = -\frac{2}{3|g|} < 0.
    \end{align*}
\end{proof}

\subsection{Zero Curvature with Positive Density in Free Group}

The generating set $S_\beta$ from Proposition \ref{negative curvature example} did not allow for points of non-negative curvature because different geodesic words for the same group element could have no more than two different starting letters. In order to prove Theorem \ref{intro zero curvature dense}, we therefore seek a six-element generating set which allows for different geodesic words for the same group element to begin with three different generators. Let $S_\gamma \coloneqq \{a,a^{-1},b,b^{-1},ababa,a^{-1}b^{-1}a^{-1}b^{-1}a^{-1}\}$. We define the word $w=ababa$ so that $w^{-1} = a^{-1}b^{-1}a^{-1}b^{-1}a^{-1}$, and we split $S_\gamma$ into the two subsets $S_\gamma^1 \coloneqq \{a,b^{-1},w\}$ and $S_\gamma^2 \coloneqq \{a^{-1},b,w^{-1}\}$. Then we have the following useful results:

\begin{lem} \label{first letter lemma}
    Let $h_1\ldots h_n$ and $h_1'\ldots h_n'$ be geodesic words spelling the same element of $F_2$. Then $h_1\in S_\gamma^1$ if and only if $h_1'\in S_\gamma^1$. Similarly, $h_n\in S_\gamma^1$ if and only if $h_n'\in S_\gamma^1$.
\end{lem}
\begin{proof}
    Let the sets $A_a$, $A_{a^{-1}}$, $A_b$, and $A_{b^{-1}}$ be as in the proof of Proposition \ref{negative curvature example}, and recall that every non-identity element of $F_2$ must lie in exactly one of these four sets. Let $g\in A_a$, and we claim that every geodesic spelling of $g$ begins with an element of $S_\gamma^1$. Let $g_1\ldots g_n$ be a geodesic spelling of $g$ such that $g_1\neq a$ and $g_1\neq w$. Then there must exist $1\leq k<n$ such that $g_1\ldots g_k\notin A_a\cup \{1\}$ and $(g_1\ldots g_k)g_{k+1}\in A_a$. The only non-identity elements that are not in $A_a$ for which right multiplication by a generator gives elements of $A_a$ are $b^{-1}a^{-1}$ and $b^{-1}a^{-1}b^{-1}a^{-1}$. For $b^{-1}a^{-1}$, we have $(b^{-1}a^{-1})w=aba$, while for $b^{-1}a^{-1}b^{-1}a^{-1}$, we have $(b^{-1}a^{-1}b^{-1}a^{-1})w=a$. In this second case, though, note that if $g_1\ldots g_k=b^{-1}a^{-1}b^{-1}a^{-1}=aw^{-1}$ and $g_{k+1}=w$, then the spelling $g_1\ldots g_n$ would not have been geodesic and we would have a contradiction. Therefore, the only way we could have $g_1\ldots g_n\in A_a$ and $g_1\notin A_a$ is if $g_1\ldots g_k=b^{-1}a^{-1}$ and $g_{k+1}=w$. The word $b^{-1}a^{-1}$ is the only geodesic spelling of itself, so we must in particular have $g_1=b^{-1}$. We can conclude that either $g_1\in A_a$, in which case it is either $a$ or $w$, or $g_1=b^{-1}$. Therefore, $g_1\in S_\gamma^1$ as desired. An identical argument shows that, if $g\in A_{a^{-1}}$, then every geodesic spelling of $g$ must begin with an element of $S_\gamma^2$.
    
    We now claim that, if $g\in A_b$, then every geodesic spelling of $g$ begins with an element $S_\gamma^2$. Let $g_1\ldots g_n$ be a geodesic spelling of $g$ such that $g_1\neq b$. Then there must exist $1\leq k<n$ such that $g_1\ldots g_k\notin A_b\cup \{1\}$ and $(g_1\ldots g_k)g_{k+1}\in A_b$. The only non-identity elements not in $A_b$ for which right multiplication by a generator gives elements of $A_b$ are $a^{-1}$ and $a^{-1}b^{-1}a^{-1}$. For $a^{-1}$, we have $a^{-1}w=baba$, while for $a^{-1}b^{-1}a^{-1}$, we have $(a^{-1}b^{-1}a^{-1})w=ba$. In this second case, we would have $g_1\ldots g_k=a^{-1}b^{-1}a^{-1}=baw^{-1}$ is a geodesic spelling, and $g_{k+1}=w$, so the spelling $g_1\ldots g_n$ is not geodesic and we would have a contradiction. This leaves the case where $g_1\ldots g_k=a^{-1}$, and because this spelling is geodesic we must have $k=1$ and $g_1=a^{-1}$. Therefore, a geodesic spelling of $g\in A_b$ must begin with either $a^{-1}$ or $b$, both of which are elements of $S_\gamma^2$. An identical argument shows that, if $g\in A_{b^{-1}}$, then every geodesic spelling of $g$ must begin with an element of $S_\gamma^1$, either $b^{-1}$ or $a$.
    
    Now suppose we have two geodesic words $h_1\ldots h_n$ and $h_1'\ldots h_n'$ such that $h_1\ldots h_n= h_1'\ldots h_n' \coloneqq h$. If $h_1\in S_\gamma^1$, then the results of the previous paragraphs tell us that $h\in A_a$ or $A_{b^{-1}}$. Therefore, we must also have $h_1'\in S_\gamma^1$. If $h_1\notin S_\gamma^1$, then $h_1\in S_\gamma^2$, so $h\in A_{a^{-1}}$ or $A_b$ and $h_1'\in S_\gamma^2$. Finally, applying these results to $h^{-1}$ gives the desired result for $h_n$ and $h_n'$.
\end{proof}

\begin{lem} \label{one letter offset}
    With respect to the generating set $S_\gamma$, there do not exist generators $s\in S_\gamma$ or group elements $g\in F_2-\{1\}$ such that $|sg|=|g|$ or $|gs|=|g|$.
\end{lem}
\begin{proof}
    With respect to $S_\gamma$, $F_2$ has the presentation $\langle a,b,w|ababaw^{-1}=1\rangle$. By Lemma \ref{preserving lemma}, it suffices to show that no word of these generators has exactly one more letter than a geodesic word. Suppose, on the other hand, that we have $g_1\ldots g_n=h_1\ldots h_{n+1}$ for some generators $g_1,\ldots,g_n,h_1,\ldots,h_{n+1}\in S_\gamma$, where the word $g_1\ldots g_n$ is geodesic. Then $h_{n+1}^{-1}\ldots h_1^{-1}g_1\ldots g_n=1$, so this word with $2n+1$ letters can be reduced to the identity by applying free reductions and removing all occurrences of any inverse or cyclic permutation of $ababaw^{-1}$. Because all of these operations change the number of letters by an even number and $2n+1$ is odd, we have a contradiction.
\end{proof}

Lemma \ref{one letter offset} immediately implies the following when considered with Lemma \ref{shortening lemma} and Remark \ref{extending remark}.
\begin{cor} \label{extending cor}
    With respect to the generating set $S_\gamma$, we must have $|sg| = |g|-1$ or $|sg| = |g|+1$ for all $g\in F_2-\{1\}$ and $s\in S_\gamma$. Moreover, we have $|sg|= |g|-1$ if and only if there exists a geodesic spelling of $g$ beginning with $s^{-1}$. Similarly, we must have $|gs| = |g|-1$ or $|gs|=|g|+1$, and $|gs|=|g|-1$ if and only if there exists a geodesic spelling of $g$ ending with $s^{-1}$.\qed
\end{cor}

We have the following observations that use similar ideas to the proof of Lemma \ref{first letter lemma}:

\begin{lem}\label{double a}
    Suppose that, for some $g\in F_2-\{1\}$, there exists a geodesic spelling of $g$ of the form $aag_1\ldots g_n$ with $g_1,\ldots,g_n\in S_\gamma$. Then there exists no geodesic spelling of $g$ of the form $b^{-1}h_1\ldots h_{n+1}$ with $h_1,\ldots,h_{n+1}\in S_\gamma$. Similarly, if a geodesic spelling of some $g'\in F_2-\{1\}$ is of the form $a^{-1}a^{-1}g_1'\ldots g_n'$, then there exists no geodesic spelling of $g'$ of the form $bh_1'\ldots h_{n+1}'$.
\end{lem}
\begin{proof}
    In the proof of Lemma \ref{first letter lemma}, we saw that a geodesic word can only begin with the letter $a$ if it spells a group element that lies in either $A_a$ or $A_{b^{-1}}$. We cannot have $g\in A_{b^{-1}}$ in this case because this would imply that the product of the first two letters of a geodesic spelling of $g$ must lie in $A_{b^{-1}}$, but we have $aa\in A_a$. Thus, $g\in A_a$. If a geodesic spelling $b^{-1}h_1\ldots h_{n+1}$ of $g$ exists, then we must have $h_1=a^{-1}$ and $h_2=w$ because we saw in the proof of Lemma \ref{first letter lemma} that all geodesic words in $A_a$ that begin with $b^{-1}$ must begin with the prefix $b^{-1}a^{-1}w$. Therefore, because $b^{-1}a^{-1}w = aba$, we have
    \begin{gather*}
        aag_1\ldots g_n=g=b^{-1}a^{-1}wh_3\ldots h_{n+1}=abah_3\ldots h_{n+1} \implies\\ ag_1\ldots g_n = bah_3\ldots h_{n+1},
    \end{gather*}
    which contradicts Lemma \ref{first letter lemma} because we have equal geodesic words beginning with $a\in S_\gamma^1$ and $b\notin S_\gamma^1$. The argument for $g'$ is identical.
\end{proof}

\begin{lem} \label{b double a}
    Let $g\in F_2-\{1\}$ such that there exists a geodesic spelling of $g$ of the form $baag_1\ldots g_n$ with $g_1,\ldots,g_n\in S_\gamma$. Then there exists no geodesic spelling of $g$ of the form $a^{-1}h_1\ldots h_{n+2}$ with $h_1,\ldots,h_{n+2}\in S_\gamma$. Similarly, if $g'\in  F_2-\{1\}$ has a geodesic spelling of the form $b^{-1}a^{-1}a^{-1}g_1'\ldots g_n'$, then there exists no geodesic spelling of $g'$ of the form $ah_1'\ldots h_{n+2}'$.
\end{lem}
\begin{proof}
    Because $aag_1\ldots g_n=b^{-1}g$ is in $A_a$, we must have that $baag_1\ldots g_n = g$ is in $A_b$. Therefore, if a geodesic spelling $a^{-1}h_1\ldots h_{n+2}$ of $g$ exists, then we must have $h_1=w$ because we saw in the proof of Lemma \ref{first letter lemma} that all geodesic words in $A_b$ that begin with $a^{-1}$ must begin with the prefix $a^{-1}w$. Because $aba = wa^{-1}b^{-1}$, we have
    \begin{align*}
        &baag_1\ldots g_n = g = a^{-1}wh_2\ldots h_{n+2}\\
        \implies&
        abaag_1\ldots g_n = wh_2\ldots h_{n+2}\\
        \implies&
        wa^{-1}b^{-1}ag_1\ldots g_n=wh_2\ldots h_{n+2}\\
        \implies&
        b^{-1}ag_1\ldots g_n = ah_2\ldots h_{n+2}.
    \end{align*}
    The word $ag_1\ldots g_n$ is geodesic because it is a subword of a geodesic word for $g$. Because this word begins with $a$, there exists no other geodesic spelling representing the same group element that begins with $b$ by Lemma \ref{first letter lemma}. Thus, by Corollary \ref{extending cor}, $b^{-1}ag_1\ldots g_n$ is a geodesic word, so $ah_2\ldots h_{n+2}$ is also geodesic because it consists of the same number of generators. Note that $b^{-1}ag_1\ldots g_{n-1}$ must either lie in $A_a$ or $A_{b^{-1}}$ because it begins with $b^{-1}$. If a geodesic word beginning with $b^{-1}$ lies in $A_a$, then it must begin with the prefix $b^{-1}a^{-1}w$, so $b^{-1}ag_1\ldots g_{n-1}$ instead lies in $A_{b^{-1}}$. Therefore, $ah_2\ldots h_{n+2}$ must also lie in $A_{b^{-1}}$, which implies that $h_2=w^{-1}$ because all geodesic words in $A_{b^{-1}}$ that begin with $a$ must begin with the prefix $aw^{-1}$. We then have
    \begin{align*}
        &b^{-1}ag_1\ldots g_n = aw^{-1}h_3\ldots h_{n+2}\\
        \implies&
        wa^{-1}b^{-1}ag_1\ldots g_n = h_3\ldots h_{n+2}\\
        \implies&
        abaag_1\ldots g_n = h_3\ldots h_{n+2} \\
        \implies&
        aag_1\ldots g_n = b^{-1}a^{-1}h_3\ldots h_{n+2},
    \end{align*}
    which is a contradiction by Lemma \ref{double a}. The argument for $g'$ is identical.
    \end{proof}

The following proposition gives a sufficient condition for a point in $F_2$ to have zero curvature with respect to $S_\gamma$. In our proof of Theorem \ref{intro zero curvature dense}, we show that points satisfying this condition not only exist, but exist with positive density.

\begin{prop} \label{zero curvature}
    With respect to $S_\gamma$, $\kappa(g) = 0$ for all $g\in F_2$ such that $g$ has a geodesic spelling of the form $g=(aba)^{\pm 1}(h) (aba)^{\pm 1}$ or $g=(aba)^{\pm 1} (h) (aba)^{\mp 1}$, where $h$ is a geodesic subword such that $|h|=|g|-6$.
\end{prop}
\begin{proof}
    Suppose first that $g=(aba)(h)(aba)$. Then
    \begin{equation*}
        ga^{-1} = (aba)(h)(aba)a^{-1} = (aba)(h)(ab),
    \end{equation*}
    which must be a geodesic spelling of $|ga^{-1}|$ because it consists of $|g|-1$ letters, the minimum length of a spelling of $ga^{-1}$ by the triangle inequality. By Lemma \ref{first letter lemma}, every geodesic spelling of $ga^{-1}$ must begin with an element of $S_\gamma^1$ so no geodesic spelling begins with $a^{-1}$. Combining this fact with Corollary \ref{extending cor}, we have
    \begin{equation*}
        |aga^{-1}|= |a(aba)(h)(ab)| = |(aba)(h)(ab)|+1 = (|g|-1)+1=|g|.
    \end{equation*}
    Likewise, we have
    \begin{align*}
        |a^{-1}ga| &= |a^{-1}(aba)(h)(aba)a| = |(ba)(h)(aba)a|\\
        &= |(ba)(h)(aba)|+1 = (|g|-1)+1 = |g|.
    \end{align*}
    To compute the length of conjugations by the other generators, we use the relation $aba=wa^{-1}b^{-1}=b^{-1}a^{-1}w$ to show that
    \begin{align*}
        |bgb^{-1}| &= |b(b^{-1}a^{-1}w)(h)(aba)b^{-1}| = |(a^{-1}w)(h)(aba)b^{-1}|\\
        &= |(a^{-1}w)(h)(aba)|+1 = (|g|-1)+1 = |g|, \\
        |b^{-1}gb| &= |b^{-1}(aba)(h)(wa^{-1}b^{-1})b| = |b^{-1}(aba)(h)(wa^{-1})|\\
        &= |(aba)(h)(wa^{-1})|+1 = (|g|-1)+1 = |g|, \\
        |wgw^{-1}| &= |w(aba)(h)(b^{-1}a^{-1}w)w^{-1}| = |w(aba)(h)(b^{-1}a^{-1})| \\
        &= |(aba)(h)(b^{-1}a^{-1})|+1 = (|g|-1)+1 = |g|, \\
        |w^{-1}gw| &= |w^{-1}(wa^{-1}b^{-1})(h)(aba)w| = |(a^{-1}b^{-1})(h)(aba)w| \\
        &= |(a^{-1}b^{-1})(h)(aba)|+1 = (|g|-1)+1 = |g|.
    \end{align*}
    Therefore, $\GenCon(g) = \frac{1}{|S_\gamma|}\sum_{s\in S_\gamma}|s^{-1}gs| = |g|$, so $\kappa(g) = 0$. 
    
    To prove the case of $g=(aba)^{-1}(h)(aba)^{-1}$, note that $g^{-1}$ is a group element of the type considered in the previous case. Therefore, because the word length of a group element is equal to the length of its inverse, we have $|s^{-1}gs| = |s^{-1}g^{-1}s| = |g^{-1}| = |g|$ for all $s\in S_\gamma$. Therefore, $\GenCon(g) = \frac{1}{|S_\gamma|}\sum_{s\in S_\gamma}|s^{-1}gs| = |g|$, so $\kappa(g) = 0$.
    
    Now suppose that we have $g=(aba)(h)(aba)^{-1}$. Then, using the relations $aba = b^{-1}a^{-1}w = wa^{-1}b^{-1}$ and $a^{-1}b^{-1}a^{-1} = w^{-1}ab = baw^{-1}$, we have
    \begin{gather*}
        |a^{-1}ga| = |a^{-1}(aba)(h)(a^{-1}b^{-1}a^{-1})a| = |(ba)(h)(a^{-1}b^{-1})| \leq |g|-2,\\
        |bgb^{-1}| = |b(b^{-1}a^{-1}w)(h)(w^{-1}ab)b^{-1}| = |(a^{-1}w)(h)(w^{-1}a)| \leq |g|-2,\\
        |w^{-1}gw| = |w^{-1}(wa^{-1}b^{-1})(h)(baw^{-1})w| = |(a^{-1}b^{-1})(h)(ba)| \leq |g|-2.
    \end{gather*}
    By the triangle inequality, $|s^{-1}gs| \geq |g|-2$ for any generator $s\in S_\gamma$. Therefore, each of the above conjugations has a length of exactly $|g|-2$.
    
    Because we have a geodesic spelling of $g$ beginning with $a$, no geodesic spelling of $g$ can begin with $a^{-1}$ by Lemma \ref{first letter lemma}. Combining this fact with Corollary \ref{extending cor}, we have that the spelling
    \begin{equation*}
        ag = a(aba)(h)(aba)^{-1}
    \end{equation*}
    is geodesic. This spelling ends with $a^{-1}$, so we can once again apply Lemma \ref{first letter lemma} to show that no geodesic spelling of $ag$ ends with $a$. Thus, the spelling
    \begin{equation*}
        aga^{-1} = a(aba)(h)(aba)^{-1}a^{-1}
    \end{equation*}
    is geodesic, so $|aga^{-1}| = |g|+2$. Likewise, no geodsic spelling of $g$ starts with $b$ or $w^{-1}$ or ends with $b^{-1}$ or $w$. Thus,
    \begin{gather*}
        |b^{-1}gb| = |b^{-1}(aba)(h)(aba)^{-1}b| = |g|+2,\\
        |wgw^{-1}| = |w(aba)(h)(aba)^{-1}w^{-1}| = |g|+2.
    \end{gather*}
    Therefore, $\GenCon(g) = \frac{1}{|S_\gamma|}\sum_{s\in S_\gamma}|s^{-1}gs| = \frac{1}{6}(3(|g|+2)+3(|g|-2))= |g|$, so $\kappa(g) = 0$. The proof for the case of $g=(aba)^{-1}(h)(aba)$ is identical, with $|s^{-1}gs| = |g|-2$ for $s = a^{-1}$, $b$, or $w^{-1}$, and $|s^{-1}gs| = |g|+2$ for $s = a$, $b^{-1}$, or $w$.
\end{proof}

The following proposition gives a sufficient condition for points in $F_2$ to have negative curvature with respect to $S_\gamma$, which we also show to be true with positive density in our proof of Theorem \ref{intro zero curvature dense}.

\begin{prop} \label{negative double a}
    With respect to $S_\gamma$, we have $\kappa(g) < 0$ for all $g\in F_2$ such that a geodesic spelling of $g$ has the form $g=(aa)^{\pm 1} (h) (aa)^{\pm 1}$ or $g=(aa)^{\pm 1}(h) (aa)^{\mp 1}$, where $h$ is a geodesic subword such that $|h|=|g|-4$.
\end{prop}
\begin{proof}
    Suppose first that $g = (aa)(h)(a^{-1}a^{-1})$. Then no geodesic spelling of $g$ begins with $a^{-1}$, $b$, or $w^{-1}$ by Lemma \ref{first letter lemma}, and no geodesic spelling begins with $b^{-1}$ by Lemma \ref{double a}. If $s\in \{a, b^{-1}, w, b\}$, then the word $s(aa)(h)(a^{-1}a^{-1})=sg$ is geodesic by Corollary \ref{extending cor}. Applying Lemmas \ref{first letter lemma} and \ref{double a} again, we have that no geodesic spelling of $sg$ ends with $a$, $b^{-1}$, $w$, or $b$. Therefore, if $s\in \{a, b^{-1}, w, b\}$, then the spelling $s(aa)(h)(a^{-1}a^{-1})s^{-1}=sgs^{-1}$ is geodesic by Corollary \ref{extending cor}, so $|sgs^{-1}| = |g|+2$. Because this holds for four out of the six choices for generator $s\in S_\gamma$, and the minimum possible value for $|sgs^{-1}|$ is $|g|-2$ by the triangle inequality, this implies that $\GenCon(g) = \frac{1}{|S_\gamma|}\sum_{s\in S_\gamma}|s^{-1}gs|>|g|$ so $\kappa(g)<0$ for all such $g$. The proof for $g = (a^{-1}a^{-1})(h)(aa)$ is identical, with $|sgs^{-1}| = |g|+2$ for $s\in \{a^{-1}, b^{-1}, w^{-1}, b\}$.
    
    Now suppose that $g = (aa)(h)(aa)$. Then again by Lemmas \ref{first letter lemma} and \ref{double a}, no geodesic spellings of $g$ start or end with $b$ or $b^{-1}$, so $|b^{-1}gb|= |bgb^{-1}| = |g|+2$. No geodesic spelling of $g$ begins or ends with $a^{-1}$ or $w^{-1}$ by Lemma \ref{first letter lemma}, so $|s^{-1}gs|\geq |g|$ for $s\in\{a,a^{-1},w,w^{-1}\}$. Therefore, we again have $\GenCon(g) = \frac{1}{|S_\gamma|}\sum_{s\in S_\gamma}|s^{-1}gs|>|g|$ so $\kappa(g)<0$.
    
    Finally, the proof for $g = (a^{-1}a^{-1})(h)(a^{-1}a^{-1})$ follows by noting that $|sgs^{-1}|=|sg^{-1}s^{-1}| = |g^{-1}|+2=|g|+2$ for $s=b$ or $b^{-1}$ because $g^{-1}$ is a group element of the type considered in the previous case. Likewise, we have $|sgs^{-1}|=|sg^{-1}s^{-1}| \geq |g^{-1}|=|g|$ for $s\in\{a,a^{-1},w,w^{-1}\}$. Therefore, $\GenCon(g) = \frac{1}{|S_\gamma|}\sum_{s\in S_\gamma}|s^{-1}gs|>|g|$ so $\kappa(g)<0$.
\end{proof}

In order to prove Theorem \ref{intro zero curvature dense}, we recall the following useful inequality:
\begin{lem}  \label{frac decomp}
    Let $a_1,\ldots ,a_k,b_1,\ldots ,b_k$ be positive integers such that $\frac{a_1}{b_1} \leq \ldots  \leq \frac{a_k}{b_k}$. Then $\frac{a_1}{b_1} \leq \frac{\sum_{i=1}^k a_i}{\sum_{j=1}^k b_j}$.
\end{lem}
\begin{proof}
    This follows from induction on $k$. For $k=2$, we have
    \begin{gather*}
        \frac{a_1}{b_1} \leq \frac{a_2}{b_2} \implies
        a_1b_2 \leq a_2b_1 \implies\\
        a_1b_2+a_1b_1 \leq a_2b_1+a_1b_1 \implies
        \frac{a_1}{b_1} \leq \frac{a_1+a_2}{b_1+b_2}.
    \end{gather*}
    Suppose that the desired inequality holds for $k\leq l-1$. Then we have $\frac{a_2}{b_2} \leq \frac{\sum_{i=2}^l a_i}{\sum_{j=2}^l b_j}$. By assumption, $\frac{a_1}{b_1} \leq \frac{a_2}{b_2}$. Therefore, we can apply the result for $k=2$ to conclude that
    \begin{equation*}
        \frac{a_1}{b_1} \leq \frac{\sum_{i=2}^l a_i}{\sum_{j=2}^l b_j} \implies \frac{a_1}{b_1} \leq \frac{a_1+\sum_{i=2}^l a_i}{b_1+\sum_{j=2}^l b_j} = \frac{\sum_{i=1}^l a_i}{\sum_{j=1}^l b_j}.
    \end{equation*}
\end{proof}

\begin{proof}[Proof of Theorem \ref{intro zero curvature dense}]
    Define the sets
    \begin{align*}
        P_n \coloneqq \{g \in S_n|&~\text{there exists a geodesic spelling}\\&~g=g_1\ldots g_n~
        \text{with}~g_1,g_n \in \{a,a^{-1}\}\}.
    \end{align*}
    We claim that the ratio $\frac{|P_n|}{|S_n|}$ is bounded below by some positive number for all $n>0$. For $n>2$ and $g \in S_{n-2}$, let $g_1\ldots g_{n-2}$ be a geodesic spelling of $g$. If $g_1\in S_\gamma^1$, then by Lemma \ref{first letter lemma} no geodesic spelling of $g$ begins with $a^{-1}$. Hence, by Corollary \ref{extending cor} the word $ag_1\ldots g_{n-2}\in S_{n-1}$ is geodesic. Similarly, if $g_1\in S_\gamma^2$, then the word $a^{-1}g_1\ldots g_{n-2}\in S_{n-1}$ is geodesic. Let $h_1g_1\ldots g_{n-2}$ be a geodesic word where $h_1\in \{a,a^{-1}\}$. If $g_{n-2}\in S_\gamma^1$, then no geodesic spelling of $h_1g$ ends with $a^{-1}$ by Lemma \ref{first letter lemma}, so by Corollary \ref{extending cor} the word $h_1g_1\ldots g_{n-2}a$ is geodesic. Similarly, if $g_{n-2}\in S_\gamma^2$, then $h_1g_1\ldots g_{n-2}a^{-1}$ is geodesic. Therefore, for all $g\in S_{n-2}$, there exists at least one element of $P_n$ of the form $h_1gh_2$ where $h_1,h_2\in \{a,a^{-1}\}$.
    
    Note that every element of $S_n$ for $n>2$ can be written as $h_1gh_2$ for some $h_1,h_2\in S_\gamma$ and $g\in S_{n-2}$. Because there are six generators in $S_\gamma$, there are no more than $6^2=36$ elements of $S_n$ of the form $h_1gh_2$ for each $g\in S_{n-2}$. Therefore, for all $n>2$,
    \begin{equation*}
        \frac{|P_n|}{|S_n|}\geq \frac{1}{36}.
    \end{equation*}
    $P_1$ contains the element $a$, and $P_2$ contains the element $aa$, so both are nonempty. Therefore, letting $\eta= \min\left\{\frac{1}{|B_2|},\frac{1}{36}\right\}$, we have
    \begin{equation*}
        \frac{|P_n|}{|S_n|} \geq \eta
    \end{equation*}
    for all $n>0$ as desired.
    
    Let $g\in P_n$ be such that a geodesic spelling of $g$ is $ag_1\ldots g_{n-1}$ for some generators $g_1,\ldots,g_{n-1}\in S_\gamma$. Then no geodesic spelling of $g$ begins with $a^{-1}$ by Lemma \ref{first letter lemma}, so the word $a(ag_1\ldots g_{n-1}) = ag$ is geodesic by Corollary \ref{extending cor}. By Lemma \ref{double a}, no geodesic spelling of $ag$ begins with $b^{-1}$, so the word $b(aag_1\ldots g_{n-1})=bag$ is geodesic by Corollary \ref{extending cor}. Finally, no geodesic spelling of $bag$ begins with $a^{-1}$ by Lemma \ref{b double a}, so the word $a(baag_1\ldots g_{n-1}) = abag$ is geodesic by Corollary \ref{extending cor}. Therefore, we have $|(aba)g| = n+3$.
    
    An identical argument shows that, if $g\in P_n$ is such that a geodesic spelling of $g$ is $a^{-1}g_1\ldots g_{n-1}$ for generators $g_1,\ldots, g_{n-1}\in S_\gamma$, then the word $(aba)^{-1}a^{-1}g_1\ldots g_{n-1}$ is geodesic with length $n+3$. Applying these claims to the inverses of elements of $P_n$, we get that we either have $|g(aba)|=n+3$ or $|g(aba)^{-1}|=n+3$ for all $g\in P_n$. We can thus conclude that, if a geodesic spelling of some $g\in P_n$ is $a^{\pm1} g_1\ldots g_{n-2}a^{\pm1}$, then the word $(aba)^{\pm1} (a^{\pm1} g_1\ldots g_{n-2}a^{\pm1})(aba)^{\pm1}$ is geodesic with length $n+6$, and if a geodesic spelling of $g\in P_n$ is $a^{\pm1} g_1\ldots g_{n-2}a^{\mp1}$, then the word $(aba)^{\pm1} (a^{\pm1} g_1\ldots g_{n-2}a^{\mp1})(aba)^{\mp1}$ is geodesic with length $n+6$. This fact implies a natural bijection between the sets $P_{n-6}$ and 
    \begin{align*}
        Q_n \coloneqq \{g \in S_n|&~g=(aba)^{\pm1}(h)(aba)^{\pm1}~\text{or}~g=(aba)^{\pm1}(h)(aba)^{\mp1}\\
        &~\text{is a geodesic spelling for some}~h\in P_{n-6}\}.
    \end{align*}  
    for all $n>6$. Therefore, $|P_{n-6}|=|Q_n|$. Recall that elements of $Q_n$ have zero curvature by Proposition \ref{zero curvature}.
    
    Every element of the sphere $S_n$ can be written as an element of the sphere $S_{n-1}$ that has been left-multiplied by one of the six generators in $S$. Therefore, $|S_n| \leq 6|S_{n-1}|$. For $C > 6^6$ and $n > 6$, this implies that $|S_n| < C|S_{n-6}|$. We therefore have for all $n>6$ that
    \begin{equation}  \label{Q/S}
        \frac{|Q_n|}{|S_n|} > \frac{|Q_n|}{C|S_{n-6}|} = \frac{|P_{n-6}|}{C|S_{n-6}|} \geq \frac{\eta}{C}.
    \end{equation}
    
    Let $\epsilon_1 = \min\left\{\frac{1}{|B_7|}, \frac{\eta}{C}\right\}$ and $N_1 = 7$. Applying Lemma \ref{frac decomp}, equation \ref{Q/S}, and the fact that $|Q_7| > 1$ (it contains the points $(aba)(a)(aba)$ and $(aba)^{-1}(a^{-1})(aba)^{-1}$), we have for all $n\geq N_1$ that
    \begin{align*}
         \frac{\sum_{i=7}^n|Q_i|}{|B_n|} &> \frac{1+\sum_{i=8}^n|Q_i|}{|B_7|+\sum_{j=8}^n |S_j|}\\ &\geq \min\left\{\frac{1}{|B_7|},\frac{|Q_8|}{|S_8|},\ldots ,\frac{|Q_n|}{|S_n|}\right\} \\
         &\geq \min\left\{\frac{1}{|B_7|}, \frac{\eta}{C}\right\} = \epsilon_1.
    \end{align*}
    Therefore, the set of points in the union of all the $Q_n$ has positive density in $F_2$. Because $\kappa(g) = 0$ for all such points $g$, this proves that $F_2$ has zero curvature with positive density with respect to $S_\gamma$.
    
    We now show that points with negative curvature also have positive density. For $n\geq 5$, define the sets 
    \begin{align*}
        R_n \coloneqq \{g\in S_n|&~g~\text{has a geodesic spelling of the form}\\
        &~g_1(h)g_2~\text{with}~g_1,g_2\in \{a,a^{-1}\}~\text{and}~h\in P_{n-2}\}.
    \end{align*} 
    Note that $a^{\pm1}(h)a^{\pm1}$ is geodesic if and only if $h\in P_{n-2}$ is such that a geodesic spelling of $h$ is $a^{\pm1} h_1\ldots h_{n-4}a^{\pm1}$, and $a^{\pm1}(h)a^{\mp1}$ is geodesic if and only if a geodesic spelling of $h$ is $a^{\pm1} h_1\ldots h_{n-4}a^{\mp1}$. By definition, every element of $P_{n-2}$ satisfies exactly one of these conditions, so there is a bijection between $R_n$ and $P_{n-2}$. Therefore, $|R_n|=|P_{n-2}|$. Recall that elements of $R_n$ have negative curvature for $n\geq 5$ by Proposition \ref{negative double a}.

    Let $D>6^2=36$ so that $|S_n|<D|S_{n-2}|$ for all $n>2$. Then we have
    \begin{equation} \label{R/S}
        \frac{|R_n|}{|S_n|} > \frac{|R_n|}{D|S_{n-2}|} = \frac{|P_{n-2}|}{D|S_{n-2}|} \geq \frac{\eta}{D}.
    \end{equation}
    Let $\epsilon_2= \min\left\{\frac{1}{|B_5|},\frac{\eta}{D}\right\}$ and $N_2= 5$. Then by Lemma \ref{frac decomp}, equation \ref{R/S}, and the fact that $a^5,a^{-5}\in R_5$ so $|R_5|>1$, we have for all $n\geq N_2$ that
    \begin{align*}
         \frac{\sum_{i=5}^n|R_i|}{|B_n|} &> \frac{1+\sum_{i=6}^n|R_i|}{|B_5|+\sum_{j=6}^n |S_j|}\\ &\geq \min\left\{\frac{1}{|B_5|},\frac{|R_6|}{|S_6|},\ldots ,\frac{|R_n|}{|S_n|}\right\} \\
         &\geq \min\left\{\frac{1}{|B_5|}, \frac{\eta}{D}\right\} = \epsilon_2.
    \end{align*}
    Therefore, the set of points in the union of all the $R_n$ has positive density in $F_2$. Because $\kappa(g)<0$ for all such points $g$, this proves that $F_2$ has negative curvature with positive density with respect to $S_\gamma$.
\end{proof}

\section{Relationship with $\delta$-Hyperbolicity} \label{relationship with hyperbolicity}

The group $F_2$ from Section \ref{free and virtually free} is an example of a \emph{non-elementary} hyperbolic group, which is a hyperbolic group that is not virtually cyclic. We have thus proven that non-elementary hyperbolic groups can have a positive density of points with zero curvature. In this section we prove Theorem \ref{intro negative curvature for hyperbolic} to conclude that we get negative curvature for all but finitely many points in non-elementary hyperbolic groups if we compute the $k$-spherical conjugation curvature for sufficiently large $k$. On the other hand, we show that there exist non-hyperbolic groups that have negative curvature for all non-identity points.

\subsection{Negative Curvature at Larger Radii for Hyperbolic Groups}

For the rest of this section, let $G$ be a non-elementary $\delta$-hyperbolic group that has a word metric $d(\cdot,\cdot)$ and Gromov product $(\cdot,\cdot)_\cdot$ determined by some finite generating set $S$. We begin by proving a series of lemmas that give constants to determine the desired value of $K$ in Theorem \ref{intro negative curvature for hyperbolic}. This proof is analogous to and motivated by \cite[Example 15]{Oll09} for Ollivier's transportation curvature.

\begin{lem}  \label{tree lemma}
    There exists a constant $C_1\geq0$ such that, for all integers $k\geq 1$, $g\in G-B_{4k}$, and $s \in S_k$ we have
    \begin{equation*}
        |(1,gs)_g-(s,gs)_g| \leq C_1.
    \end{equation*}
\end{lem}
\begin{proof}
    By \cite[Section 6.2]{Gro87}, there exists a tree $T$ in the Cayley graph of $(G,S)$ consisting of at most three geodesic segments $T_1$, $T_2$, and $T_3$ such that the extremal points (vertices of degree 1) of $T$ are elements of $\{1,s,g,gs\}$, and which approximates distances between elements of $\{1,s,g,gs\}$ within some error for which we can compute an upper bound. We use $d_T(\cdot,\cdot)$ and $_T(\cdot,\cdot)_\cdot$ to denote respectively the length metric and Gromov product induced on $T$ by the graph metric of the Cayley graph of $(G,S)$. Note that, because $T$ is a subset of this Cayley graph, we have $d_T(x,y)\geq d(x,y)$ for all $x,y\in T$. By the triangle inequality,
    \begin{align*}
        (s,gs)_1+(s,gs)_g
        =& \frac{1}{2}(d(s,1)+d(1,gs)-d(s,gs)+d(s,g)+d(g,gs)-d(s,gs)) \\
        \leq& \frac{1}{2}(k+(|g|+k)-(|g|-2k)+(|g|+k)+k-(|g|-2k|)\\
        =& 4k< |g|.
    \end{align*}
    Therefore, we have the following inequality:
    \begin{equation*}
        (s,gs)_1+(s,gs)_g < |g|+12\delta = d(1,g) + 4(\log_2(4)+1)\delta.
    \end{equation*}
    By Step 2 in \cite[Section 6.2]{Gro87}, this inequality implies that we can let $T_1$ be a geodesic segment connecting $1$ and $g$, and let $T_2$ and $T_3$ be two of the shortest geodesic segments connecting $s$ and $gs$ respectively to points on $T_1$ if $s$ and $gs$ do not already lie on $T_1$ (if they do then $T$ will consist of only one or two geodesic segments). By Property 3 in \cite[Section 6.2]{Gro87}, we have for all $x,y\in \{1,s,g,gs\}$ that
    \begin{equation*}
        d_T(x,y) \leq d(x,y)+C\delta(\log_2(4))^2 = d(x,y)+4C\delta
    \end{equation*}
    for some $C \leq 100$. Because $d_T(x,y)\geq d(x,y)$, we can rewrite this statement as
    \begin{equation}  \label{tree ineq}
        |d_T(x,y)-d(x,y)|\leq 4C\delta.
    \end{equation}
    Because $d(1,s) = d(g,gs) = k$ and $d(1,g) > 4k$, $T_2$ and $T_3$ do not intersect and the intersection of $T_3$ and $T_1$ is closer to $g$ than that of $T_2$ and $T_1$. Therefore, $_T(1,gs)_g=~_T(s,gs)_g$ (if $gs\in T_1$, then both of these products are equal to the distance from $g$ to $gs$). The tree $T$ is shown in Figure \ref{tree-approx}.
    \begin{figure}[h]
        \centering
        \includegraphics[scale = 0.22]{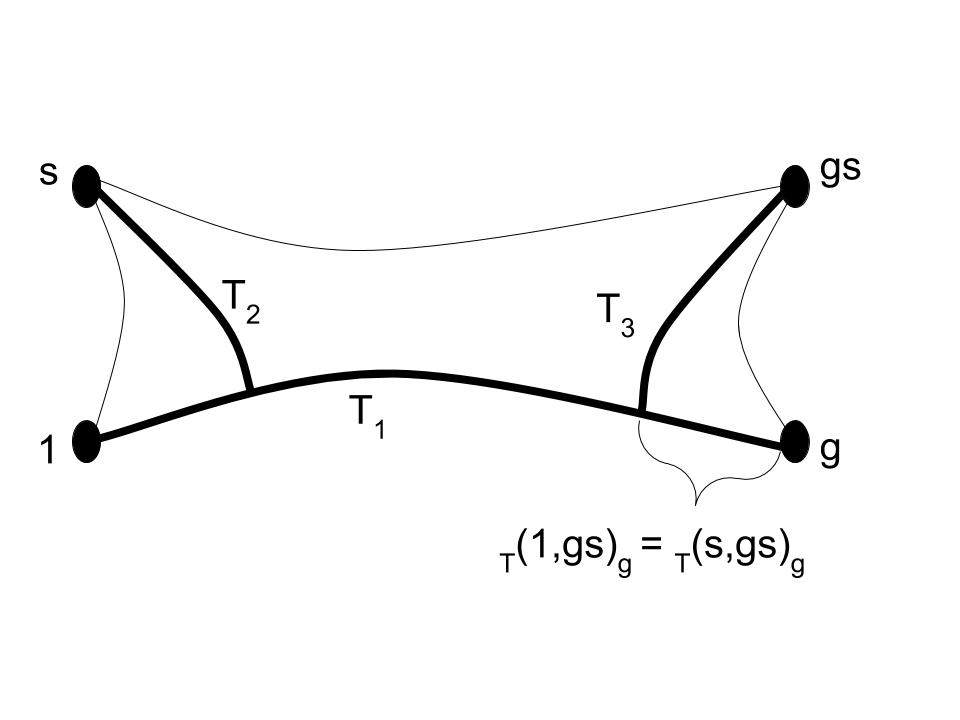}
        \caption{A geodesic tree (bold segments) connecting $1$, $s$, $g$, and $gs$ such that the induced Gromov products $_T(1,gs)_g$ and $_T(s,gs)_g$ are equal.}
        \label{tree-approx}
    \end{figure}    
    
    By the equation $_T(1,gs)_g=~_T(s,gs)_g$, inequality \ref{tree ineq}, and the triangle inequality,
    \begin{align*}
        &|(1,gs)_g-(s,gs)_g| \\
        =& |[~_T(1,gs)_g-~_T(s,gs)_g] + [(1,gs)_g-~_T(1,gs)_g] + [~_T(s,gs)_g-(s,gs)_g]| \\
        \leq& |(1,gs)_g-~_T(1,gs)_g|+|~_T(s,gs)_g-(s,gs)_g| \\
        \leq& \frac{1}{2}(|d(1,g)-d_T(1,g)|+|d(g,gs)-d_T(g,gs)|+|d_T(1,gs)-d(1,gs)| \\
        &~+|d(s,g)-d_T(s,g)|+|d(g,gs)-d_T(g,gs)|+|d_T(s,gs)-d(s,gs)|) \\
        \leq& \frac{1}{2}6(4C\delta) \leq 1200\delta.
    \end{align*}
    Therefore, letting $C_1 = 1200\delta$ gives the desired result.
\end{proof}
 
The following is a well-known fact for non-elementary hyperbolic groups. See, for example, \cite[Corollary IV.1.3.]{Cha12} for a proof.
\begin{lem}  \label{growth lemma}
    There exists a constant $C_2 > 1$ such that, for all $k\geq 1$, we have
    \begin{equation*}
        |B_k| \geq C_2|B_{k-1}|.
    \end{equation*}\qed
\end{lem}

We have the following lemma that shows that a majority of geodesic curves connecting points in a ball to points far away will pass close to the center of the ball. The proof is analogous to \cite[Proposition 21]{Oll04}.
\begin{lem} \label{ball bounding lemma}
    There exists a constant $C_3>0$ such that, for all $g\in G$, $k\geq 1$, and integers $L$ with $2\lceil\delta\rceil\leq L\leq k$ we have
    \begin{equation*}
        \frac{|A_k(g,L)|}{|B_k|} \leq C_3C_2^{-L},
    \end{equation*}
    where
    \begin{equation*}
        A_k(g,L) \coloneqq \{s\in B_k|(g,s)_1\geq L\},
    \end{equation*}
    and $C_2$ is the constant from Lemma \ref{growth lemma}.
\end{lem}
\begin{proof}
    Fix any such $g$, $k$, and $L$, and let $s \in B_k$. Suppose that $(g,s)_1 \geq L$. Then by the triangle inequality, $(g,s)_1 = \frac{1}{2}(|g|+|s|-d(g,s)) \leq \frac{1}{2}(|g|+|s|-(|g|-|s|)) = |s|$, so $L \leq |s|$, and we similarly have $L \leq |g|$. Let $h,h'\in G$ respectively be points on a geodesic segment connecting $1$ and $g$ and a geodesic segment connecting $1$ and $s$ such that $d(1,h) = d(1,h') = L$. Applying the $2\delta$-thin condition to a geodesic triangle with vertices $1$, $g$, and $s$ containing $h$ and $h'$, we have that $d(h,s)\leq d(h',s)+d(h,h') \leq (k-L)+2\delta$ (see Figure \ref{gromov to ball} on page \pageref{gromov to ball}). Therefore, $A_k(g,L)\subset B_{k-L+2\lceil\delta\rceil}(h)$. 
    
    By left-invariance of the word metric, it suffices to show that $\frac{|B_{k-L+2\lceil\delta\rceil}|}{|B_k|}$ is bounded above by the desired function of $L$. Note that $k-L+2\lceil\delta\rceil$ is an integer bounded above by $k$. Therefore, letting $C_2$ be the constant from Lemma \ref{growth lemma}, we have
    \begin{gather*}
        C_2|B_{k-1}| \leq |B_k| \implies C_2^{L-2\lceil\delta\rceil}|B_{k-L+2\lceil\delta\rceil}| \leq |B_k| \implies \\
        \frac{|B_{k-L+2\lceil\delta\rceil}|}{|B_k|} \leq C_2^{-L+2\lceil\delta\rceil} \implies \frac{|B_{k-L+2\lceil\delta\rceil}|}{|B_k|} \leq C_3C_2^{-L}.
    \end{gather*}
    where $C_3 \coloneqq C_2^{2\lceil\delta\rceil}$. The fact that $C_3$ does not depend on $g$, $k$, or $L$ implies the desired result.
\end{proof}
\begin{figure}
    \centering
    \includegraphics[scale = 0.27]{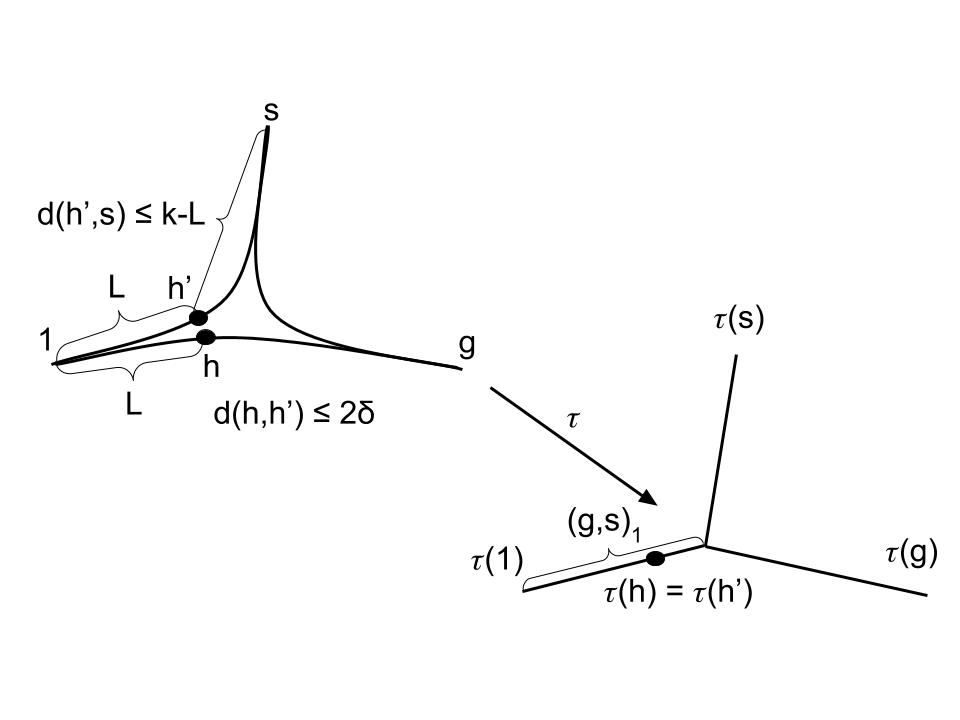}
    \caption{Using the $2\delta$-thin condition and triangle inequality to show that, because $(g,s)_1 \geq L$, we must have $d(h,s)\leq k-L+2\delta$.}
    \label{gromov to ball}
\end{figure}

The following lemma uses the previous result to deduce an upper bound on the average value of the Gromov product between all the points on a sphere and any arbitrary point:
\begin{lem}  \label{sphere bounding lemma}
    There exists a constant $C_4 > 0$ such that for all $g\in G$ and $k\geq 1$ we have
    \begin{equation*}
        \frac{1}{|S_k|}\sum_{s\in S_k}(g,s)_1 \leq C_4.
    \end{equation*}
\end{lem}
\begin{proof}
    Define the set
    \begin{equation*}
        D_k(g,L) \coloneqq \{s\in S_k|(g,s)_1 = L\}.
    \end{equation*}
    Then we have
    \begin{equation*}
        \frac{1}{|S_k|}\sum_{s\in S_k}(g,s)_1 = \sum_{L=0}^k L\frac{|D_k(g,L)|}{|S_k|}
    \end{equation*}
    because each $(g,s)_1 \leq |s| = k$ and all the $(g,s)_1$ are non-negative integers. Let $C_2$ be the constant from Lemma \ref{growth lemma}. Then
    \begin{equation*}
        \frac{|S_k|}{|B_k|} = \frac{|B_k|-|B_{k-1}|}{|B_k|}
        = 1 - \frac{|B_{k-1}|}{|B_k|}
        \geq 1-\frac{1}{C_2}
        \eqqcolon b>0.
    \end{equation*}
    Also, define $a \coloneqq \sum_{L=0}^{2\lceil \delta\rceil-1}L$, which is an upper bound for $\sum_{L=0}^{2\lceil \delta\rceil-1}L\frac{|D_k(g,L)|}{|S_k|}$ for all $k$ because $\frac{|D_k(g,L)|}{|S_k|}\leq 1$. Then because $D_k(g,L) \subset A_k(g,L)$, we can apply Lemma \ref{ball bounding lemma} to conclude that
    \begin{align*}
        \sum_{L=0}^k L\frac{|D_k(g,L)|}{|S_k|} &\leq a+\sum_{L=2\lceil \delta\rceil}^k L\frac{|D_k(g,L)|}{|S_k|} \leq a+\sum_{L=2\lceil \delta\rceil}^k L\frac{|A_k(g,L)|}{|S_k|}\\
        &\leq a+\sum_{L=2\lceil \delta\rceil}^k L\frac{|A_k(g,L)|}{b|B_k|}
        \leq a+\frac{C_3}{b}\sum_{L=0}^k LC_2^{-L},
    \end{align*}
    where $C_3$ is the constant from Lemma \ref{ball bounding lemma}. For $L\geq\frac{1}{\ln(C_2)}$, the sequence $(LC_2^{-L})_{L=0}^\infty$ is monotonically decreasing. Letting
    \begin{gather*}
        f \coloneqq \left\lfloor\frac{1}{\ln(C_2)}\right\rfloor, \\
        c \coloneqq \left\lceil\frac{1}{\ln(C_2)}\right\rceil, \\
        I \coloneqq \int_c^\infty xC_2^{-x}dx < \infty,
    \end{gather*}
    we have by the integral test that
    \begin{equation*}
        a+\frac{C_3}{b}\sum_{L=0}^k LC_2^{-L}
        < a+\frac{C_3}{b}\sum_{L=0}^\infty LC_2^{-L}
        \leq a+\frac{C_3}{b}\left(\sum_{L=0}^fLC_2^{-L} + cC_2^{-c}+I\right) \eqqcolon C_4,
    \end{equation*}
    where $C_4$ does not depend on $g$ or $k$. 
\end{proof}

The constants $C_1$ and $C_4$ can now be applied to prove our main result:
\begin{proof}[Proof of Theorem \ref{intro negative curvature for hyperbolic}]
    Fix $k\geq1$ and let $g \in G-B_{4k}$. For all $s \in S_k$,
    \begin{align*}
        d(s,gs) =& d(1,g)+d(1,s)+d(g,gs) \\
        &~-[d(g,1)+d(1,s)-d(g,s)] -[d(s,g)+d(g,gs)-d(s,gs)] \\
        =& |g|+2k-2(g,s)_1-2(s,gs)_g \\
        =& |g|+2k-2(g,s)_1-2(1,gs)_g+2[(1,gs)_g-(s,gs)_g].
    \end{align*}
    The curvature is then given by
    \begin{align*}
        \kappa_k^{\mathcal{S}}(g)
        =& \frac{|g|-\frac{1}{|S_k|}\sum_{s\in S_k}d(s,gs)}{|g|} \\
        =& \frac{|g|-\frac{1}{|S_k|}\sum_{s\in S_k}|g|+2k-2(g,s)_1-2(1,gs)_g + 2[(1,gs)_g-(s,gs)_g]}{|g|} \\
        =& \frac{-2k+\left(\frac{1}{|S_k|}\sum_{s\in S_k}2(g,s)_1+2(1,gs)_g - 2[(1,gs)_g-(s,gs)_g]\right)}{|g|},
    \end{align*}
    so
    \begin{align*}
        |g|\kappa_k^{\mathcal{S}}(g) +2k
        =& 2\left(\sum_{s\in S_k}\frac{(g,s)_1}{|S_k|} + \sum_{s\in S_k}\frac{(1,gs)_g}{|S_k|} - \frac{1}{|S_k|}\sum_{s\in S_k}[(1,gs)_g-(s,gs)_g]\right) \\
        \leq& 2(C_4+C_4+C_1).
    \end{align*}
    where $C_1$ and $C_4$ are the constants from Lemmas \ref{tree lemma} and \ref{sphere bounding lemma} respectively (the inequality for the $(1,gs)_g$ term follows from using left-invariance to rewrite it as $(g^{-1},s)_1$ and applying Lemma \ref{sphere bounding lemma}). We recall that the constants $C_1$ and $C_4$ are independent of $k$. Therefore, letting $K = \lceil2C_4+C_1+1\rceil$, we have for all $k\geq K$ and $g\in G-B_{4k}$ that
    \begin{equation*}
        \kappa_k^{\mathcal{S}}(g) \leq \frac{2(C_4+C_4+C_1)-2k}{|g|} < 0.
    \end{equation*}
\end{proof}

\subsection{Negatively Curved Non-Hyperbolic Groups}

With our proof of Theorem \ref{intro negative curvature for hyperbolic}, we have shown that hyperbolicity implies negative conjugation curvature for sufficiently large comparison radius. It is natural to ask whether negative curvature everywhere conversely implies that a group is hyperbolic. We conclude this paper by proving that this is not the case.

\begin{proof}[Proof of Theorem \ref{intro negative non-hyperbolic}]
    Gromov proved in \cite{Gro87} that no hyperbolic group contains $\Z^2$ as a subgroup. Therefore, $\Z*\Z^2$ is not a hyperbolic group.
    
    Using the presentation $\Z*\Z^2 = \langle a,b,t|ab=ba\rangle$, let $S_\text{neg}\coloneqq\{a,a^{-1},b,b^{-1},t,t^{-1}\}$, and we will show that we always have negative curvature with respect to $S_\text{neg}$. Suppose that $g_1\ldots g_n = h_1\ldots h_{n+1}$ for some generators $g_1,\ldots,g_n,h_1,\ldots, h_{n+1}\in S_\text{neg}$. Then $h_{n+1}^{-1}\ldots h_1^{-1}g_1\ldots g_n=1$, so this word with $2n+1$ letters can be reduced to the identity by applying free reductions and commuting the letters $a$, $a^{-1}$, $b$, and $b^{-1}$ with each other. Because free reductions reduce the number of letters by an even number and $2n+1$ is odd, we have a contradiction. Therefore, given any geodesic word of generators in $S_\text{neg}$, there does not exist a word with exactly one more generator that represents the same group element in $\Z*\Z^2$. Thus, by Lemma \ref{preserving lemma} we do not have $|sg| = |g|$ or $|gs|=|g|$ for any $g \in \Z*\Z^2$ or $s\in S_\text{neg}$.
    
    Let $g\in\Z*\Z^2-\{1\}$. From the definition of $\Z*\Z^2-\{1\}$, we see that a certain geodesic spelling of $g$ contains the letter $t^{\pm 1}$ in the $i$th position if and only if every geodesic spelling of $g$ does. Likewise, a geodesic spelling of $g$ contains an element of $\{a,a^{-1},b,b^{-1}\}$ in the $i$th position if and only if every geodesic spelling of $g$ contains an element of this set in the $i$th position, and no geodesic spelling of $g$ can contain $s$ in the $i$th position for $s\in \{a,a^{-1},b,b^{-1}\}$ if another geodesic spelling of $g$ contains $s^{-1}$ in the $i$th position.
    
    Suppose first that every geodesic spelling of $g$ does not include the letters $t$ or $t^{-1}$. Then no geodesic spelling of $tg$ or $t^{-1}g$ will end with $t$ or $t^{-1}$, so $|tgt^{-1}|=|t^{-1}gt|=|g|+2$ by applying Remark \ref{extending remark} twice. On the other hand, all the letters of a geodesic spelling of $g$ commute with the letters in $\{a,a^{-1},b,b^{-1}\}$, so for all $s\in\{a,a^{-1},b,b^{-1}\}$ we have $|s^{-1}gs| = |gs^{-1}s| = |g|$. Adding these together, we get that $\GenCon(g) = \frac{1}{6}\sum_{s\in S_\text{neg}}|s^{-1}gs| = |g|+\frac{4}{6}$ and $\kappa(g) = -\frac{2}{3|g|}<0$.

    Next suppose that every geodesic spelling of $g$ starts and ends with elements of $\{a,a^{-1},b,b^{-1}\}$ and also contains at least one element of $\{t,t^{-1}\}$. Then we still have $|tgt^{-1}|=|t^{-1}gt|=|g|+2$. There exist at most two generators $s\in\{a,a^{-1},b,b^{-1}\}$ such that $|sg|=|g|-1$, and at most two such that $|gs|=|g|-1$. Because every geodesic spelling of $g$ contains a letter $t$ or $t^{-1}$ that does not commute with the letters in $\{a,a^{-1},b,b^{-1}\}$, we have $|s^{-1}gs|=|gs|-1$ if and only if $|s^{-1}g| = |g|-1$. Therefore,
    \begin{align*}
        \kappa(g) &= \frac{|g|-\GenCon(g)}{|g|} = \frac{|g|-\frac{1}{6}\sum_{s\in S_\text{neg}}|s^{-1}gs|}{|g|} \\
        &= \frac{\frac{1}{6}\sum_{s\in S_\text{neg}}(|g|-|gs|) + \frac{1}{6}\sum_{s\in S_\text{neg}}(|gs|-|s^{-1}gs|)}{|g|}\\ &\leq \frac{\frac{1}{6}(-2)+\frac{1}{6}(-2)}{|g|} = -\frac{2}{3|g|} < 0.
    \end{align*}
    
    If every geodesic spelling of $g$ starts and ends with elements of $\{t,t^{-1}\}$, then for all $s\in\{a,a^{-1},b,b^{-1}\}$ we have $|s^{-1}gs| = |g|+2$. Therefore, because the minimum possible value of $|s^{-1}gs|$ for any $s\in S_\text{neg}$ is $|g|-2$ by the triangle inequality, we have $\GenCon(g) = \frac{1}{6}\sum_{s\in S_\text{neg}}|s^{-1}gs| \geq |g|+\frac{4}{6}$ and $\kappa(g) \leq -\frac{2}{3|g|}<0$.
    
    Finally, suppose that every geodesic spelling of $g$ either starts or ends with $t^{\pm 1}$ and ends or starts with an element of $\{a,a^{-1},b,b^{-1}\}$. Without loss of generality, say that every geodesic spelling of $g$ starts with $t$, and that at least one geodesic spelling of $g$ ends with $a$. Then we must have
    \begin{align*}
        |tgt^{-1}| &= |g|+2,\\
        |a^{-1}ga| &= |g|+2,\\
    \end{align*}
    and $|s^{-1}gs|\geq |g|$ for all $s\in \{t,a^{-1},b,b^{-1}\}$. Therefore, we have $\GenCon(g) = \frac{1}{6}\sum_{s\in S_\text{neg}}|s^{-1}gs| \geq |g|+\frac{4}{6}$ and $\kappa(g) \leq -\frac{2}{3|g|}<0$.
    
    We can thus conclude that $\kappa(g)\leq -\frac{2}{3|g|}<0$ for all $g\in \Z_2*\Z_3-\{1\}$. Therefore, $\Z_2*\Z_3$ is a non-hyperbolic group that has negative curvature for all non-identity points.
    \end{proof}

\printbibliography

\end{document}